\documentclass[reqno]{amsart}
\usepackage{amssymb}
\usepackage[hypertex]{hyperref}

\usepackage{verbatim}

\newcommand{\mysection}[1]{\section{#1}
\setcounter{equation}{0}}

\newtheorem{theorem}{Theorem}[section]
\newtheorem{corollary}[theorem]{Corollary}
\newtheorem{lemma}[theorem]{Lemma}
\newtheorem{proposition}[theorem]{Proposition}

\theoremstyle{definition}
\newtheorem{remark}[theorem]{Remark}

\theoremstyle{definition}

\theoremstyle{definition}
\newtheorem{assumption}[theorem]{Assumption}

\makeatletter
\def\dashint{\operatorname%
{\,\,\text{\bf--}\kern-.98em\DOTSI\intop\ilimits@\!\!}}
\makeatother

\def\vg{\textit{\textbf{g}}}

\def\bR{\mathbb{R}}

\def\bH{\mathbb{H}}

\def\fp{\mathfrak{p}}

\def\cA{\mathcal{A}}

\def\cH{\mathcal{H}}

\def\cU{\mathcal{U}}

\newcommand\rV{\mathring{V}}
\newcommand{\ip}[1]{\left\langle#1\right\rangle}
\newcommand{\set}[1]{\left\{#1\right\}}
\newcommand{\norm}[1]{\lVert#1\rVert}

\newcommand{\abs}[1]{\left\lvert#1\right\rvert}
\providecommand{\tri}[1]{\lvert\lVert#1\rvert\rVert}
\newcommand{\Div}{\operatorname{div}}

\begin{document}

\title[Quasilinear elliptic and parabolic equations]{Global regularity of weak solutions to quasilinear elliptic and parabolic equations with controlled growth}

\author[H. Dong]{Hongjie Dong}
\address[H. Dong]{Division of Applied Mathematics, Brown University,
182 George Street, Providence, RI 02912, USA}
\email{Hongjie\_Dong@brown.edu}
\thanks{H. Dong was partially supported by NSF grant number DMS-0800129.}

\author[D. Kim]{Doyoon Kim}
\address[D. Kim]{Department of Applied Mathematics, Kyung Hee University,
1, Seochun-dong, Gihung-gu, Yongin-si, Gyeonggi-do 446-701 Korea}
\email{doyoonkim@khu.ac.kr}



\begin{abstract}
We establish global regularity for weak solutions to quasilinear divergence form elliptic and parabolic equations over Lipschitz domains with controlled growth conditions on low order terms.
The leading coefficients belong to the class of BMO functions with small mean oscillations with respect to $x$.
\end{abstract}

\maketitle

\mysection{Introduction}

The paper is devoted to the study of the global regularity of weak solutions to quasilinear divergence form elliptic equations on Lipschitz domains with the Dirichlet boundary condition:
\begin{equation}							 \label{eq3.24}
\left\{
  \begin{aligned}
    -D_i\left(A_{ij}(x,u) D_j u + a_i(x,u)\right) = b(x,u,\nabla u)
 \quad & \hbox{in $\Omega$,} \\
    u=0 \quad & \hbox{on $\partial \Omega$,}
  \end{aligned}
\right.
\end{equation}
and quasilinear divergence form parabolic equations on cylindrical domains with the Cauchy-Dirichlet boundary condition:
\begin{equation}							 \label{eq3.30}
\left\{
  \begin{aligned}
    u_t-D_i\left(A_{ij}(t,x,u) D_j u + a_i(t,x,u)\right) = b(t,x,u,\nabla u)
 \quad & \hbox{in $\cU_T$,} \\
    u=0 \quad & \hbox{on $\partial_p \cU_T$.}
  \end{aligned}
\right.
\end{equation}
Here $\Omega\subset\bR^d$ is a bounded Lipschitz domain with a small Lipschitz constant, $\cU_T=(0,T)\times \Omega$,
and $\partial_p \cU_T = \left([0,T) \times \partial\Omega\right) \cup \left(\{0\}\times \Omega\right)$.

The nonlinear terms $A_{ij}(t,x,u)$, $a_i(t,x,u)$, and $b(t,x,u,\xi)$
in the parabolic equation \eqref{eq3.30} are of Carathe\'{o}dory type, i.e., they are measurable in $(t,x) \in \bR\times\bR^d$ for all $(u,\xi) \in \bR \times \bR^d$,
and continuous in $(u,\xi) \in \bR \times \bR^d$ for almost all $(t,x) \in \bR\times\bR^d$.
The terms in the elliptic equation \eqref{eq3.24} are also of Carathe\'{o}dory type with no time variable.
The leading coefficients $A_{ij}$ are bounded and uniformly elliptic, that is, for some constant $\mu\in (0,1]$,
$$
|A_{ij}| \le \mu^{-1},
\quad
A_{ij}\xi_i \xi_j \ge \mu |\xi|^2\quad \forall\xi\in \bR^d.
$$
We also assume that $A_{ij}(\cdot,x,u)$ are uniformly continuous in $u$ and have small mean oscillations with respect to $x$.
It is well-known that functions in this class are not necessarily continuous.
In the parabolic case, the coefficients $A_{ij}$ are further allowed to be merely measurable in the time variable.

The lower order terms $a_i$ and $b$ in \eqref{eq3.30} satisfy the following {\em controlled growth conditions}:
$$
|a_i(t,x,u)| \le \mu_1 (|u|^{\lambda_1} + f),
$$
$$
|b(t,x,u,\nabla u)| \le \mu_2 (|\nabla u|^{\lambda_2} + |u|^{\lambda_3} + g),
$$
for some constants $\mu_1,\mu_2>0$, where
$\lambda_1= \frac{d+2}{d}$, $\lambda_2= \frac{d+4}{d+2}$, $\lambda_3= \frac{d+4}{d}$,
and
$$
f \in L_{\sigma}(\cU_T),
\quad
g\in L_{\tau}(\cU_T),
\quad
\sigma\in (d+2,\infty),
\quad
\tau\in (d/2+1,\infty).
$$
We impose similar conditions in the elliptic case; see Section \ref{main}.
The controlled growth conditions guarantee that weak solutions to the equations \eqref{eq3.24}
and \eqref{eq3.30} are well-defined (see an explanation above Theorem \ref{thm1}).
If $\lambda_i$, $i=1,2,3$, are strictly less than the
numbers above,
we say that the equations satisfy 
{\em strictly controlled growth conditions}.

Under the above assumptions, in this paper we prove that weak solutions to the quasilinear equations \eqref{eq3.24} and \eqref{eq3.30} have higher global integrability. For instance, a weak solution to \eqref{eq3.30} is proved to be a member of $\cH_p^1(\cU_T)$ (see Section \ref{sec1} for the definition of the $\cH_p^1$ space),
where $p>d+2$ is determined only by the integrability of $f$ and $g$ (i.e., $\sigma$ and $\tau$).
As an easy consequence, by the Sobolev embedding theorem,
weak solutions turn out to be {\em globally} H\"{o}lder continuous with the H\"{o}lder exponents depending only on the dimension and the integrability 
of $f$ and $g$.

There has been tremendous work on the regularity of weak solutions to divergence type elliptic and parabolic equations/systems.
Let us mention some of them to explain our results here.
For linear equations, the fundamental results by De Giorgi \cite{De Giorgi} and Nash \cite{Nash} show the interior H\"{o}lder regularity of weak solutions.
For linear or quasilinear/nonlinear systems, to which weak solutions are in general partially regular, there has been a lot of discussions on higher integrability of solutions and reverse H\"older's inequalities (see, for instance, \cite{Pulv67,Gehring,Camp,Gi83}), which are the key ingredients in the proofs of partial regularity results.
When quasilinear or nonlinear equations/systems are considered, the regularity of weak solutions has been investigated under various growth conditions on lower order terms.
By nonlinear systems we mean here, in the elliptic case, equations of the form $\Div A(x,u,Du) = b(x,u,Du)$.

More specifically, with linear ($\lambda_2=1$), quadratic ($\lambda_2=2$), or (strictly) controlled growth conditions imposed on the lower order nonlinear terms, various reverse H\"{o}lder's inequalities, and partial regularities of weak solutions to quasilinear or nonlinear systems have been obtained in \cite{MeyElc,Gi78,GiaMod, Camp1982, Gi83, Camp1986, DuGr00} and \cite{Camp1981, Camp1982P,GiaStr, Camp1984, MarinoMaugeri}.
In particular, linear growth conditions for parabolic systems were considered in \cite{Camp1981} and \cite{GiaStr}, where the latter one also considered a quadratic growth condition.
The strictly controlled growth conditions were investigated in \cite{Camp1982P, Camp1984, MarinoMaugeri} for parabolic systems. In \cite{Gi78,Camp1986,DuGr00} elliptic systems with quadratic growth conditions were considered.
The controlled growth conditions for elliptic systems were investigated in \cite{Camp1982} and \cite{Gi83}.
In \cite{GiaMod} the authors considered three different kinds of growth conditions including the controlled and quadratic growth conditions.
We remark that the quadratic growth conditions are always accompanied by an additional smallness assumption on solutions.

The corresponding boundary estimates are more delicate. Under the controlled growth conditions, Arkhipova investigated {\em Neumann problems} for divergence type quasilinear elliptic and parabolic systems, for example, in \cite{Arkh94, Arkh95} (see references therein), where she proved reverse H\"older's inequality and partial regularities up to the boundary of solutions. The key steps are careful boundary estimates using the structure of Neumann boundary conditions.
To the best of our knowledge, the corresponding results for the Dirichlet problem \eqref{eq3.30} are not available in the literature.
We also mention that, under a quadratic growth condition, similar results (i.e., partial regularity up to the boundary) for quasilinear elliptic systems with non-homogeneous Dirichlet boundary conditions were obtained later in \cite{Grotowski}.
Regarding general nonlinear homogeneous parabolic systems (i.e. $b=0$), very recently B\"ogelein, Duzaar and Mingione \cite{BoDuMi10} obtained boundary partial H\"older regularity of $Du$ for the Cauchy-Dirichlet problem; see also \cite{DuKrMi07} for a corresponding result for elliptic systems.
Note that in general global regularity cannot be expected for systems (see \cite{Gi_ex,StJoMa86}), and even for partial regularities usually one requires the leading coefficients to possess certain regularity in all involved variables (usually uniform continuity).

Recently regularity theory for quasilinear equations with {\em discontinuous} coefficients has been studied in \cite{FengZheng,Pala09}.
In \cite{FengZheng}, Feng and Zheng established an interior reverse H\"older's inequality for quasilinear elliptic systems with the controlled growth conditions under the assumption that the leading coefficients are in the class of VMO functions with respect to $x$ variables.
In addition, 
they obtained the optimal interior H\"{o}lder continuity of solutions to scalar equations as well as partial H\"{o}lder regularity of solutions to systems.
With the same VMO assumption on the leading terms, Palagachev \cite{Pala09} proved the {\em global} H\"{o}lder regularity of solutions to elliptic quasilinear equations in $C^1$ domains.
He used a bootstrap argument which, however, requires the strictly controlled growth conditions.
The relaxation of the regularity assumptions on the leading coefficients from uniform continuity to VMO in \cite{FengZheng, Pala09} relies on the $L_p$-theory of linear equations/systems with VMO coefficients, the study of which was initiated in \cite{CFL1}. For quasilinear nondivergence form equations with discontinuous coefficients, we refer the reader to the book \cite{MaPaSo00} and reference therein.

In view of the more general growth conditions on coefficients and the global nature of the H\"{o}lder regularity in this paper, our results can be considered as generalizations of the known regularity results for weak solutions to quasilinear divergence form elliptic and parabolic equations with the Dirichlet boundary condition. In particular, we generalize the results in \cite{Pala09} for elliptic as well as parabolic equations under the controlled growth conditions. 
As noted earlier, in the parabolic case we do not require any regularity of the terms in the equation as functions of the time variable.
The controlled growth conditions are optimal (see, for instance, a counterexample in \cite{Pala09}) unless some additional boundedness conditions on weak solutions are imposed.
It is worth noting that
in the parabolic case with $d=1,2$ the growth conditions defined in this paper are more general than those commonly used before; see, for example, \cite{Camp1984}. This is because we use a multiplicative inequality (Lemma \ref{lem5.1}) instead of the Sobolev imbedding theorem, which is not optimal when $d=1,2$.

As we mentioned above, the bootstrap argument in \cite{Pala09} cannot be applied directly to weak solutions under the controlled growth conditions.
To achieve our main results, we first establish reverse H\"older's inequalities (slightly higher integrability) for quasilinear equations \eqref{eq3.24} and \eqref{eq3.30} under the controlled growth conditions. It should be mentioned that the reverse H\"older's inequalities, which have their own interest, also hold for systems.
The main difficulty here is to have the interior and boundary estimates in the same form in order to apply Gehring-Giaquinta-Modica's lemma. It turns out that in our case the proof of the interior estimate for parabolic equations is more involved (see Proposition \ref{prop01}). The slightly higher integrality enables us to go through the bootstrap argument shown in \cite{Pala09}, by utilizing the recent development of $L_p$-theory for divergence form linear equations with BMO coefficients (for instance, see \cite{DK09} and references therein).
Note that, as in \cite{Pala09},  we have explicit descriptions of H\"{o}lder exponents in terms of the summability of $f$ and $g$, whereas such explicit H\"{o}lder exponents are not shown in the De Giorgi-Moser-Nash theory. It is worth mentioning that by using the slightly higher integrability results established in this paper and the arguments, for example, in \cite{Gi83,Arkh94,FengZheng}, one may also obtain the partial regularity up to the boundary of weak solutions to systems. We do not intend to pursue this in the current paper.

This paper is organized as follows. In Section \ref{sec1} we introduce some notation and definitions. Then we state our main results in Section \ref{main}.
Section \ref{sec4} is devoted to reverse H\"{o}lder's inequalities for the parabolic case, which are obtained by boundary and interior estimates.
In Section \ref{sec7} we present some $L_p$-theory for linear equations in order to run the bootstrap argument in Section \ref{sec6}, where higher integrability of solutions is proved, thus the global H\"{o}lder regularity follows. In the last Section \ref{elliptic-sec} we briefly treat the elliptic case.

\mysection{Notation and definitions}
								\label{sec1}

We use $X=(t,x)$ to denote a point in $\bR^{d+1}$; $x=(x^1,\ldots, x^d)$ will always be a point in $\bR^d$. We define the parabolic distance between two points $X=(t,x)$ and $Y=(s,y)$ in $\bR^{d+1}$ as
\[\abs{X-Y}_p:=\max(\sqrt{\abs{t-s}},\abs{x-y}),\]
where $\abs{\,\cdot\,}$ denotes the usual Euclidean norm.

For a given function $u=u(t,x)$ defined on $Q \subset \bR^{d+1}$, we use
$D_i u$ for $\partial u/\partial x^i$,  while we use $u_t$ for $\partial u/\partial t$. For $\alpha \in (0,1]$, we define
\[
\abs{u}_{\alpha/2,\alpha;Q}=\abs{u}_{0;Q}+[u]_{\alpha/2,\alpha;Q}: = \sup_{X\in Q}\,\abs{u(X)}+\sup_{\substack{X, Y \in Q\\ X\neq Y}} \frac{\abs{u(X)-u(Y)}}{\abs{X-Y}_p^\alpha}.
\]
By $C^{\alpha/2,\alpha}(Q)$ we denote the set of all bounded measurable functions $u$ on $Q$ for which $\abs{u}_{\alpha/2,\alpha;Q}$ is finite.
We use the following notation for parabolic cylinders in $\bR^{d+1}$:
\begin{align*}
Q_r(X)=(t-r^2,t)\times B_r(x),
\end{align*}
where $B_r(x)$ is the usual Euclidean ball of radius $r$ centered at $x\in \bR^d$.
For an open set $\Omega\subset\bR^{d}$, we set
\[
\Omega_r (x) =  \Omega\cap B_r (x).
\]
For an open set $\cU\subset\bR^{d+1}$, we set
\begin{equation*}
\cU_r (X) =  \cU\cap Q_r (X).	
\end{equation*}
We write $\cU(t_0)$ for the set of all points  $(t_0,x)$ in $\cU$ and $I(\cU)$  for the set of all $t$ such that $\cU(t)$ is nonempty.
For a function $u$ defined on $\cU$, we occasionally use the following norm:
\[
\tri{u}_{\cU}^2=\|Du\|_{L_2(\cU)}^2+\sup_{t\in I(\cU)}\, \|u(t,\cdot)\|_{L_2(\cU(t))}^2.
\]

Now let $\cU:=\cU_T^S$ be the cylinder $(S,T)\times\Omega$, where $-\infty<S<T<\infty$ and $\Omega$ is a bounded domain in $\bR^d$.
Throughout the paper, as in \eqref{eq3.30} we write $\cU_T$ when $S=0$.
We denote by $W^{0,1}_2(\cU)$ the Hilbert space with the inner product
\[
\ip{u,v}_{W^{0,1}_2(\cU)}:=\int_{\cU} uv+\sum_{k=1}^d \int_{\cU} D_k u D_k v.
\]
Set $W^{1,1}_2(\cU)$ to be the subspace of $W^{0,1}_2(\cU)$ such that $u_t \in L_2(\cU)$, and $V_2(\cU)$ to be the Banach space consisting of all elements of $W^{0,1}_2(\cU)$ having a finite norm $\norm{u}_{V_2(\cU)}:= \tri{u}_{\cU}$.
By $\rV_2(\cU)$ we mean the set of all functions $u$ in $V_2(\cU)$ that vanishes on the lateral boundary $\partial_l \cU:=(S,T)\times \partial\Omega$ of $\cU$. From a well known Sobolev-like imbedding theorem (see e.g., \cite[\S II.3]{LSU}),
we have
\begin{equation} \label{eqn:2.2}
\norm{u}_{L_{2+4/d}(\cU)} \leq N(d)\, \tri{u}_{\cU}, \quad\forall u\in \rV_2(\cU).
\end{equation}

We denote $\bH^{-1}_{p}(\cU)$ to be the space consisting of all functions $u$ satisfying
$$
\inf\set{\|\vg\|_{L_{p}(\cU)}+\|h\|_{L_{p}(\cU)}\,|\,u=\Div \vg+h}<\infty.
$$
It is easy to see that $\bH^{-1}_{p}(\cU)$ is a Banach space. Naturally, for any $u\in \bH^{-1}_{p}(\cU)$, we define the norm
\begin{equation*}
\|u\|_{\bH^{-1}_{p}(\cU)}=\inf\set{\|\vg\|_{L_{p}(\cU)}+\|h\|_{L_{p}(\cU)}\,|\,
u=\Div \vg+h}.
\end{equation*}
We also define
$$
\cH^{1}_{p}(\cU)=
\set{u:\,u,Du \in L_{p}(\cU),u_t\in \bH^{-1}_{p}(\cU)}.
$$
Note that $\cH_2^1(\cU) \subset V_2(\cU)$.

Let $a\wedge b=\min(a,b)$. Finally, by $N(d,p,\cdots)$ we mean that $N$ is a constant depending only
on the prescribed quantities $d, p,\cdots$.

\mysection{Main results}
							\label{main}
							
We first introduce a bounded Lipschitz domain $\Omega$ which we use throughout the paper.
A constant $\beta$ will be specified later.

\begin{assumption}[$\beta$]
                                    \label{assump2}
There is a constant $R_0\in (0,1]$ such that, for any $x_0\in \partial\Omega$ and $r\in(0,R_0]$, there exists a Lipschitz
function $\phi$: $\bR^{d-1}\to \bR$ such that
$$
\Omega\cap B_r(x_0) = \{x \in B_r(x_0)\, :\, x^1 >\phi(x')\}
$$
and
$$
\sup_{x',y'\in B_r'(x_0'),x' \neq y'}\frac {|\phi(y')-\phi(x')|}{|y'-x'|}\le \beta
$$
in an appropriate coordinate system.
\end{assumption}

Let us first state our results for elliptic equations,
assuming the following controlled growth conditions 
on the lower order terms:
$$
|a_i(x,u)| \le \mu_1 (|u|^{\gamma/2} + f),
$$
$$
|b(x,u,\nabla u)| \le \mu_2 (|\nabla u|^{2(1-1/\gamma)} + |u|^{\gamma-1} + g),
$$
for some constants $\mu_1,\mu_2>0$, where
$$
f\in L_2(\Omega),\quad g\in L^{\frac \gamma {\gamma-1}}(\Omega),\quad
\gamma =
\left\{
\begin{aligned}
&\frac{2d}{d-2}, \quad d > 2,\\
&\text{any number bigger than $2$}, \quad d = 2.
\end{aligned}
\right.
$$
We call $u\in \mathring{W}^1_2(\Omega)$ a weak solution to \eqref{eq3.24} if for any $\phi\in \mathring{W}^1_2(\Omega)$, we have
$$
\int_\Omega (A_{ij}(x,u)D_ju+a_i(x,u)) D_i\phi\,dx=\int_\Omega b(x,u,Du)\phi\,dx.
$$
Note that the controlled growth conditions are natural conditions because they are needed for the convergence of the integrals in the definition of weak solutions above.

\begin{theorem}[Reverse H\"older's inequality for elliptic equations]
                                        \label{thm1}
Let $u\in \mathring{W}^1_2(\Omega)$ be a weak solution to \eqref{eq3.24}. Suppose in addition that $f\in L_{\sigma}(\Omega)$ and $g\in L_{\tau}(\Omega)$ for some $\sigma\in (2,\infty)$ and $\tau\in (\gamma/(\gamma-1),\infty)$. Then there exists $p>2$ depending only on $d$, $\mu$, $\mu_1$,
$\mu_2$, $\sigma$, $\tau$, $\beta$, and $u$, such that
$$
\|u\|_{L_{\gamma p/2}(\Omega)} + \|Du\|_{L_p(\Omega)}\le N,
$$
where $N = N(d,\mu,\mu_1,\mu_2,\sigma,\tau, \beta, u, \|f\|_{L_\sigma(\Omega)},
\|g\|_{L_\tau(\Omega)}, |\Omega|)$.
\end{theorem}

To get the optimal global regularity for the equation \eqref{eq3.24}, we need a few more assumptions.
Let
$$
A^{\#}_R = \sup_{1 \le i,j \le d}\sup_{\substack{x_0 \in \bR^d\\ z_0 \in \bR, r \le R}}\dashint_{B_r(x_0)}\dashint_{B_r(x_0)} |A_{ij}(x,z_0) - A_{ij}(y,z_0)| \, dx \, dy.
$$

The following assumption indicates that $A_{ij}(x,\cdot)$ have small mean oscillations as functions of $x\in\bR^d$.

\begin{assumption}[$\rho$]
								\label{SMO}
There is a constant $R_1 \in (0,1]$ such that $A^{\#}_{R_1} \le \rho$.
\end{assumption}

We also need a continuity assumption on $A_{ij}(\cdot,z)$ as functions of $z\in\bR$.

\begin{assumption}
								\label{Aconti}
There exists a continuous nonnegative function $\omega(r)$ defined on $[0,\infty)$
such that $\omega(0) = 0$ and
$$
|A_{ij}(x_0,z_1) - A_{ij}(x_0,z_2)|
\le \omega\left(|z_1 - z_2|\right)
$$
for all $x_0 \in \bR^d$
and $z_1, z_2 \in \bR$.
\end{assumption}

Set
$$
q^* = \left\{
\begin{aligned}
&\frac{qd}{d-q} \quad &\text{if} \quad q < d,\\
&\text{arbitrary large number} > 1 \quad &\text{if} \quad q \ge d.	
\end{aligned}
\right.
$$
Note that if $q<d$, then $1/q^*=1/q-1/d$.

\begin{theorem}[Optimal global regularity for elliptic equations]
                                        \label{thm2}
Let $u\in \mathring{W}^1_2(\Omega)$ be a weak solution to \eqref{eq3.24}. Suppose in addition that $f\in L_{\sigma}(\Omega)$ and $g\in L_{\tau}(\Omega)$ for some $\sigma\in (d,\infty)$ and $\tau\in (d/2,\infty)$.
Then there exist positive $\beta = \beta(d, \mu, \sigma, \tau)$
and $\rho = \rho(d,\mu, \sigma, \tau)$ such that,
under Assumption \ref{assump2}($\beta$) and Assumption \ref{SMO}($\rho$),
we have
$$
\|u\|_{W^1_p(\Omega)} \le N,
\quad
\text{where}
\quad
p=\min\{\sigma,\tau^*\}>d
$$
and
$N=N(d,\mu,\mu_1,\mu_2,\sigma,\tau,\beta, u, \|f\|_{L_\sigma(\Omega)},
\|g\|_{L_\tau(\Omega)}, R_1, \omega, |\Omega|)$.
Consequently, we have $u\in C^\alpha(\bar\Omega)$ where $\alpha=1-d/p$.
\end{theorem}

Now we state our results for the parabolic equation \eqref{eq3.30},
assuming the following controlled growth conditions on the lower order terms:
$$
|a_i(t,x,u)| \le \mu_1 (|u|^{\gamma/2} + f),
$$
$$
|b(t,x,u,\nabla u)| \le \mu_2 (|\nabla u|^{2(1-1/\gamma)} + |u|^{\gamma-1} + g),
$$
where
$$
f\in L_2(\cU_T),\quad g\in L^{\frac \gamma {\gamma-1}}(\cU_T),\quad
\gamma=\frac{2(d+2)}{d}.
$$

Our first main result for parabolic equations is a reverse H\"older's inequality for the following parabolic equation. Note that we do not impose the zero initial condition as in the equation \eqref{eq3.30}.

\begin{equation}							 \label{eq0218}
\left\{
  \begin{aligned}
    u_t-D_i\left(A_{ij}(t,x,u) D_j u + a_i(t,x,u)\right) = b(t,x,u,\nabla u)
 \quad & \text{in} \,\,\cU_T,\\
    u=0 \quad & \text{on} \,\, \partial_l \cU_T.
  \end{aligned}
\right.
\end{equation}

By a weak solution to the above equation we mean $u\in \rV_2(\cU_T)$
such that, for any $\varphi \in \mathring{W}^{1,1}_2(\cU_T)$ and $t \in [0,T]$, we have
\begin{multline*}
\int_{\Omega} u(t,x) \varphi(t,x) \, dx- \int_{\Omega}u(0,x) \varphi(0,x)\,dx\\
= \int_0^t\int_{\Omega} \left[u \varphi_t - A_{ij}(s,x,u) D_ju D_i \varphi
- a_i(s,x,u) D_i \varphi + b(s,x,u,Du) \varphi\right]\,dx\,ds.
\end{multline*}

\begin{theorem}[Reverse H\"older's inequality for parabolic equations]
                                        \label{thm3}
Let $u\in \rV_2(\cU_T)$ be a weak solution to \eqref{eq0218}. Suppose in addition that $f\in L_{\sigma}(\cU_T)$ and $g\in L_{\tau}(\cU_T)$ for some $\sigma\in (2,\infty)$ and $\tau\in (\gamma/(\gamma-1),\infty)$. Then there exists $p>2$ depending only on
$d$, $\mu$, $\mu_1$, $\mu_2$, $\sigma$, $\tau$, $\beta$, $u$, $\|f\|_{L_\sigma(\cU_T)}$ and $\|g\|_{L_\tau(\cU_T)}$,
such that $u\in \cH^1_{p, \text{loc}}(\cU_T)$.
Moreover, for any $0 < \varepsilon < T$, we have
$$
\|u\|_{L_{\gamma p/2}((\varepsilon,T)\times\Omega)}
+ \|Du\|_{L_p((\varepsilon,T)\times\Omega)}\le N,
$$
where $N = N(d,\mu,\mu_1,\mu_2,\sigma,\tau,\beta, u, \|f\|_{L_\sigma(\cU_T)}, \|g\|_{L_\tau(\cU_T)}, \varepsilon, T, |\Omega|)$.
In particular, if the initial condition is zero, one can take $\varepsilon = 0$.
\end{theorem}

\begin{remark}
The statements of Theorems \ref{thm1} and \ref{thm3} are true for elliptic and parabolic systems under the same conditions.
\end{remark}

The assumption below reads that the coefficients $A_{ij}(t,x,z)$ for parabolic equations are merely measurable in $t \in \bR$ and have small mean oscillations in $x\in\bR^d$. We set
$$
\cA_R^{\#} = \sup_{1 \le i,j \le d}\sup_{\substack{(t_0,x_0) \in \bR^{d+1}\\ z_0 \in \bR, r \le R}}\dashint_{\!t_0-r^2}^{\,\,\,t_0}\dashint_{B_r(x_0)}\dashint_{B_r(x_0)} |A_{ij}(s,x,z_0) - A_{ij}(s,y,z_0)| \, dx \, dy \, ds.
$$

\begin{assumption}[$\rho$]
								\label{SMOPara}
There is a constant $R_1 \in (0,1]$ such that $\cA^{\#}_{R_1} \le \rho$.
\end{assumption}

The following is a continuity assumption on $A_{ij}(\cdot,\cdot,z)$ as functions of $z\in\bR$.
\begin{assumption}
								\label{AcontiPara}
There exists a continuous nonnegative function $\omega(r)$ defined on $[0,\infty)$
such that $\omega(0) = 0$ and
$$
|A_{ij}(t_0,x_0,z_1) - A_{ij}(t_0,x_0,z_2)|
\le \omega\left(|z_1 - z_2|\right)
$$
for all $(t_0,x_0) \in \bR^{d+1}$
and $z_1, z_2 \in \bR$.
\end{assumption}

For the parabolic case, set
$$
q^* = \left\{
\begin{aligned}
&\frac{q(d+2)}{d+2-q} \quad &\text{if} \quad q < d+2,\\
&\text{arbitrary large number} > 1 \quad &\text{if} \quad q \ge d+2.	
\end{aligned}
\right.
$$
Note that if $q < d+2$, then $1/\tau^*=1/\tau-1/(d+2)$.

\begin{theorem}[Optimal global regularity for parabolic equations]
                                        \label{thm4}
Let $u\in \mathring{\cH}^1_2(\cU_T)$ be a weak solution to \eqref{eq3.30}. Suppose in addition that $f\in L_{\sigma}(\cU_T)$ and $g\in L_{\tau}(\cU_T)$ for some $\sigma\in (d+2,\infty)$ and $\tau\in (d/2+1,\infty)$.
Then there exist positive $\beta = \beta(d, \mu, \sigma, \tau)$
and $\rho = \rho(d,\mu, \sigma, \tau)$ such that,
under Assumption \ref{assump2}($\beta$) and Assumption \ref{SMOPara}($\rho$),
we have
\begin{equation}
								\label{eq0216}
\|u\|_{\cH^1_p(\cU_T)}\le N,
\quad
\text{where}
\quad
p=\min\{\sigma,\tau^*\}>d+2
\end{equation}
and $N=N(d,\mu,\mu_1,\mu_2,\sigma,\tau,\beta, u,
\|f\|_{L_\sigma(\cU_T)}, \|g\|_{L_{\tau}(\cU_T)}, R_1, \omega, T, |\Omega|)$.
Consequently, we have $u\in C^{\alpha/2,\alpha}(\bar{\cU}_T)$ where $\alpha=1-(d+2)/p$.
\end{theorem}

\begin{remark}
Even if the initial condition of a solution $u$ to the equation \eqref{eq3.30} is not zero, we still have the same result as in \eqref{eq0216}, but with $(\varepsilon,T)\times\Omega$ in place of $\cU_T$.
This follows easily from the proof of Theorem \ref{thm4} together with an appropriate set of cut-off functions with respect to the time variable.
\end{remark}


\section{Reverse H\"{o}lder's inequality: Parabolic case}
							\label{sec4}

This section is devoted to the proof of Theorem \ref{thm3},
which is a reverse H\"{o}lder's inequality for parabolic equations.
The main ingredients of the proof are a boundary estimate (Proposition \ref{prop02}) and an interior estimate (Proposition \ref{prop01}).
For these estimates we use the following well-known multiplicative inequality; see, for example,  \cite[\S II.3]{LSU} or \cite[\S VI.3]{Lieberman}.
Throughout this section as well as Sections \ref{sec7} and \ref{sec6}, as in Theorem \ref{thm3} we set
$$
\gamma=\frac{2(d+2)}{d}.
$$

\begin{lemma}
                            \label{lem5.1}
For any $R>0$ and $u\in W^1_2(B_R)$, we have
\begin{equation*}
\|u\|_{L_\gamma(B_R)}\le N(d)\|u\|_{L_2(B_R)}^{\frac 2 {d+2}}
(\|Du\|_{L_2(B_R)}+R^{-1}\|u\|_{L_2(B_R)})^{\frac d {d+2}}.
\end{equation*}
\end{lemma}

We also use the following cut-off functions in the proofs below.
Let $\eta_0\in C_0^\infty(B_1)$ and $\zeta_0\in C_0^\infty(-1,1)$ be functions satisfying $0 \le \eta_0,\zeta_0 \le 1$, $\eta_0 \equiv 1$ on $B_{1/2}$ and $\zeta_0 \equiv 1$ on $(-1/2,1/2)$. Let $R\in (0,1]$, $x_0\in \bar\Omega$ and $t_0\in (0,T)$. Define
\begin{equation}
								\label{eq0201}
\eta(x)=\eta_0(R^{-1}(x-x_0)),\quad \zeta(t)=\zeta_0(R^{-2}(t-t_0)).								 
\end{equation}
Recall that
$\cU_r (X) =  \cU_T \cap Q_r (X)$. 
First we prove the following energy type inequality.

\begin{lemma}
							\label{lemm0201}
Let $u\in \rV_2(\cU_T)$ be a weak solution to \eqref{eq0218}, and $f\in L_2(\cU_T)$ and $g\in L_{\frac{\gamma}{\gamma-1}}(\cU_T)$.
Then, for any $x_0 \in \bR^{d}$, $0 < R \le R_0$ and $t_0 \in [R^2 ,T]$,
\begin{multline}
								\label{eq0202}
\sup_{t \in (t_0-R^2,t_0)}\int_{\Omega_R(x_0)} u^2(t,x) \eta^4(x) \zeta^2(t) \, dx
+ \mu \int_{\cU_R(X_0)} |\nabla u|^2 \eta^4 \zeta^2
\\
\le N\Big(R^{-2} \int_{\cU_R(X_0)} |u|^2\eta^2 \zeta
+ \int_{\cU_R(X_0)} |u|^{\gamma}
+ \int_{\cU_R(X_0)}(|f|^2+|g|^{\frac \gamma {\gamma-1}})\Big),	
\end{multline}
where  $X_0 = (t_0,x_0)$ and $N = N(d,\mu,\mu_1,\mu_2)$.
\end{lemma}

\begin{proof}
As a test function, multiply both sides by $u(t,x) \eta^4(x) \zeta^2(t)$.\footnote{To be rigorous, here one needs to take the Steklov average of $u$ and then pass to the limit.}
Then
\begin{align*}
&\frac 1 2\int_{\Omega_R(x_0)} u^2(t,x) \eta^4(x) \zeta^2(t) \, dx
+\int_{t_0-R^2}^t\int_{\Omega_R(x_0)} A_{ij}(s,x,u) D_j u D_i(u \eta^4 \zeta^2)\\
&\,=\int_{t_0-R^2}^t\int_{\Omega_R(x_0)} b(s,x,u,\nabla u) u \eta^4 \zeta^2
- \int_{t_0-R^2}^t\int_{\Omega_R(x_0)} a_i(s,x,u) D_i(u \eta^4 \zeta^2)\\
&\,\,\,+\int_{t_0-R^2}^t\int_{\Omega_R(x_0)} u^2 \eta^4 \zeta \zeta_t.
\end{align*}
From this we obtain
\begin{align*}
&\frac 1 4\sup_{t \in (t_0-R^2,t_0)}\int_{\Omega_{R}(x_0)} u^2(t,x) \eta^4(x) \zeta^2(t) \, dx
+ \frac \mu 2 \int_{\cU_R(X_0)} |\nabla u|^2 \eta^4 \zeta^2\\
&\,\le 4\int_{\cU_R(X_0)} |A_{ij}| |\nabla u| |u| \eta^3 |\nabla \eta| \zeta^2
+ \int_{\cU_R(X_0)} |a_i| |\nabla u| \eta^4 \zeta^2\\
&\,\,\,+ 4\int_{\cU_R(X_0)} |a_i| |u| |\nabla \eta| \eta^3 \zeta^2
+ \int_{\cU_R(X_0)} |b| |u| \eta^4 \zeta^2
+ \int_{\cU_R(X_0)} u^2 \eta^4 |\zeta_t| \zeta\\
&\,:= J_1 + J_2 + J_3 + J_4 + J_5.
\end{align*}
We estimate $J_i$, $i=1,\cdots,5$ by using Young's inequality as follows.

\noindent
Estimate of $J_1$:
$$
J_1 \le \frac \mu {16} \int_{\cU_R(X_0)} |\nabla u|^2 \eta^4 \zeta^2
+ N \int_{\cU_R(X_0)} |u|^2 |\nabla \eta|^2 \eta^2 \zeta^2.
$$

\noindent
Estimate of $J_2$:
\begin{align*}
J_2 &\le N\int_{\cU_R(X_0)} |u|^{\gamma/2} |\nabla u| \eta^4 \zeta^2
+ N\int_{\cU_R(X_0)} |f| |\nabla u| \eta^4 \zeta^2\\
&\le \frac \mu {16} \int_{\cU_R(X_0)} |\nabla u|^2 \eta^4 \zeta^2
+ N \int_{\cU_R(X_0)} |u|^{\gamma} \eta^4 \zeta^2
+ N \int_{\cU_R(X_0)} |f|^2 \eta^4 \zeta^2.
\end{align*}

\noindent
Estimate of $J_3$:
\begin{align*}
J_3 &\le N\int_{\cU_R(X_0)} |u|^{\gamma/2} |u||\nabla \eta| \eta^3 \zeta^2
+ N\int_{\cU_R(X_0)} |f| |u| |\nabla \eta| \eta^3 \zeta^2\\
&\le N\int_{\cU_R(X_0)} |u|^{\gamma} \eta^4 \zeta^2
+ N\int_{\cU_R(X_0)} |u|^2 |\nabla \eta|^2 \eta^2 \zeta^2
+ N\int_{\cU_R(X_0)} |f|^2 \eta^4 \zeta^2.
\end{align*}

\noindent
Estimate of $J_4$:
\begin{align*}
J_4 &\le N\int_{\cU_R(X_0)} |\nabla u|^{2(1-1/\gamma)} |u| \eta^4 \zeta^2
+ N\int_{\cU_R(X_0)} |u|^{\gamma} \eta^4 \zeta^2
+ N\int_{\cU_R(X_0)} |g| |u|  \eta^4 \zeta^2\\
&\le \frac \mu {16}\int_{\cU_R(X_0)} |\nabla u|^2 \eta^4 \zeta^2
+ N\int_{\cU_R(X_0)} |u|^{\gamma}
+ N\int_{\cU_R(X_0)} |g|^{\frac \gamma {\gamma-1}}.
\end{align*}

By combining the above estimates
and using the fact that $\zeta^2 \le \zeta$, we obtain
\eqref{eq0202}.
\end{proof}

As a consequence of the above inequality we prove that
$\|u(t,\cdot)\|_{L_2(\Omega_R(x))} \to 0$ uniformly in $(t,x)$ as $R \to 0$.
Especially, if $f\in L_{\sigma}(\cU_T)$ and $g\in L_{\tau}(\cU_T)$ for some $\sigma\in (2,\infty)$ and $\tau\in (\gamma/(\gamma-1),\infty)$,
by H\"{o}lder's inequality it follows that
the smallness of $\|u(t,\cdot)\|_{L_2(\Omega_R(x))}$ depends only on $u$, $\|f\|_{L_\sigma(\cU_T)}$, $\|g\|_{L_\tau(\cU_T)}$, and $R$.

\begin{corollary}
								\label{cor0201}
Under the same assumptions as in Lemma \ref{lemm0201}, we have								
\begin{equation}
                    \label{eq20.52}
\sup_{t \in (t_0-R^2/4,t_0)}\int_{\Omega_{R/2}(x_0)} u^2(t,x)\, dx\to0\quad \text{as}\,\,R\to 0
\end{equation}
uniformly in $t_0$ and $x_0$.
\end{corollary}

\begin{proof}
By H\"older's inequality
$$
R^{-2} \int_{\cU_R(X_0)} |u|^2\le N\left(\int_{\cU_R(X_0)} |u|^\gamma\right)^{\frac 2 \gamma}.
$$
From this and the inequality \eqref{eq0202} we have
\begin{multline*}
\sup_{t \in (t_0-R^2/4,t_0)}\int_{\Omega_{R/2}(x_0)} u^2(t,x)\, dx\\
\le N\int_{\cU_R(X_0)}\left( |u|^{\gamma}
+ |f|^2+ |g|^{\frac {\gamma}{\gamma-1}}\right)+N\left(\int_{\cU_R(X_0)} |u|^\gamma\right)^{\frac 2 \gamma}.
\end{multline*}
Then \eqref{eq20.52} follows from the assumptions on $u$, $f$, $g$, and \eqref{eqn:2.2}, as well as the absolute continuity of Lebesgue integrals.
\end{proof}

\subsection{Boundary case}
We now derive a reverse H\"older's inequality for solutions to \eqref{eq0218}
on a parabolic cylinder whose spatial center is located at the boundary of $\Omega$.

\begin{proposition}
								\label{prop02}
Let $u\in \rV_2(\cU_T)$ be a weak solution to \eqref{eq0218},
and $f\in L_2(\cU_T)$ and $g\in L_{\frac{\gamma}{\gamma-1}}(\cU_T)$.
Then, for any $X_0 =(t_0,x_0) \in \bR^{d+1}$ and $0 < R \le R_0$ such that $x_0 \in \partial\Omega$
and $t_0 \ge  R^2 $,
$$
\dashint_{\cU_{R/2}(X_0)} (|\nabla u|^2+|u|^\gamma)
\le N \left(\dashint_{\cU_R(X_0)}|\nabla u|^{q}\right)^{2/q}+N\dashint_{\cU_R(X_0)}(|f|^2+|F|^2)
$$
$$
\qquad+N\sup_{t \in (t_0-R^2,t_0)} \left(\int_{\Omega_{R}(x_0)} |u|^2 \, dx\right)^{\frac{2}{d}} \left(\dashint_{\cU_R(X_0)} |\nabla u|^2\right),
$$
where $q=\frac{2(d+2)}{d+4}\in (1,2)$,
$F = |g|^{\frac12\frac{\gamma}{\gamma-1}}$, and $N = N(d,\mu,\mu_1,\mu_2,\beta)$.	
\end{proposition}

\begin{proof}
We use Lemma \ref{lemm0201}. Denote by $I_1$, $I_2$, and $I_3$ the three terms on the right-hand side of the inequality \eqref{eq0202}, the estimates of which are obtained as follows.
First, note that we only need to deal with $I_1$ and $I_2$ terms.

\noindent
Estimate of $I_2$: We extend $u$ to be zero outside $\cU_T$. By Lemma \ref{lem5.1} and Poincar\'e's inequality
\begin{align*}
I_{2}&= \int_{t_0-R^2}^{t_0} \int_{B_{R}(x_0)} |u|^{\gamma}\,dx\,dt\\
&\le N\int_{t_0-R^2}^{t_0} \left(\int_{B_{R}(x_0)} |u|^2 \, dx\right)^{\frac{2}{d}}
\left(\int_{B_{R}(x_0)} |\nabla u|^2+R^{-2} | u|^2\, dx\right)\, dt\\
&\le N\sup_{t \in (t_0-R^2,t_0)} \left(\int_{\Omega_{R}(x_0)} |u|^2 \, dx\right)^{\frac{2}{d}}
\int_{\cU_R(X_0)} |\nabla u|^2 \, dx \, dt,
\end{align*}
where $N = N(d,\beta)$.

\noindent
Estimate of $I_1$: By H\"older's inequality and the Sobolev-Poincar\'e inequality,
\begin{align*}
I_1&=R^{-2} \int_{\cU_R(X_0)} |u|^2\eta^2 \zeta\\
&\le R^{-2}\int_{t_0-R^2}^{t_0} \left(\int_{B_{R}(x_0)} |u|^2 \eta^4 \zeta^2 \, dx\right)^{\frac{2}{d+4}}
\left(\int_{B_{R}(x_0)} |u|^2 \, dx\right)^{\frac{d+2}{d+4}} \, dt\\
&\le N R^{-q}\int_{t_0-R^2}^{t_0} \left(\int_{\Omega_{R}(x_0)} |u|^2 \eta^4 \zeta^2 \, dx\right)^{\frac{2}{d+4}}
\left(\int_{\Omega_{R}(x_0)} |\nabla u|^{q} \, dx\right) \, dt\\
&\le \frac \mu {16N_0} \sup_{t \in (t_0-R^2,t_0)} \int_{\Omega_{R}(x_0)} |u|^2 \eta^4 \zeta^2 \, dx+N\left(R^{-q}\int_{\cU_R(X_0)}|\nabla u|^{q}\right)^{\frac {2}{q}},
\end{align*}
where $N = N(d,\beta)$ and
$$
q=\frac{2(d+2)}{d+4}\in (1,2).
$$
Therefore, we have
\begin{align*}
\mu \int_{\cU_R(X_0)} |\nabla u|^2 \eta^4 \zeta^2
&\le R^{d+2}\left(\dashint_{\cU_R(X_0)}|\nabla u|^{q}\right)^{2/q}+N\int_{\cU_R(X_0)}(|f|^2+|F|^2)\\
&\,\,+N\sup_{t \in (t_0-R^2,t_0)} \left(\int_{\Omega_{R}(x_0)} |u|^2 \, dx\right)^{\frac{2}{d}}
\left(\int_{\cU_R(X_0)} |\nabla u|^2\right),
\end{align*}
where
$$
F =
|g|^{\frac12\frac{\gamma}{\gamma-1}}.
$$
This together with the estimate of $I_2$ yields
the inequality in the proposition.
\end{proof}

\subsection{Interior case}
This subsection is devoted to an interior version of Proposition \ref{prop02}, the proof of which is in fact more involved.

\begin{proposition}
								\label{prop01}
Let $u\in \rV_2(\cU_T)$ be a weak solution to \eqref{eq0218},
and $f\in L_2(\cU_T)$ and $g\in L_{\frac{\gamma}{\gamma-1}}(\cU_T)$.
Then, for any $X_0 \in \bR^{d+1}$ and $0 < R \le 1$ such that $B_R(x_0) \subset \Omega$
and $t_0 \ge R^2 $,
$$
\dashint_{Q_{R/2}(X_0)} (|\nabla u|^2+|u|^\gamma)
\le N \left(\dashint_{Q_R(X_0)}|\nabla u|^{q}+|u|^{\frac {\gamma q} 2}\right)^{2/q}
+N\dashint_{Q_R(X_0)}(|f|^2+|F|^2)
$$
$$
+N\sup_{t \in (t_0-R^2,t_0)} \left(\int_{B_{R}(x_0)} |u|^2 \, dx\right)^{\frac{2}{d}} \left(\dashint_{Q_R(X_0)} |\nabla u|^2\right),
$$
where
$$q=\frac {2(d+2)} {d+4},\quad F = |g|^{\frac12\frac{\gamma}{\gamma-1}},
$$ and
$N = N(d,\mu,\mu_1,\mu_2,\|u\|_{V_2(\cU_T)},\|f\|_{L_2(\cU_T)},
\|g\|_{L_{\gamma/(\gamma-1)}(\cU_T)})$.	
\end{proposition}

\begin{proof}
Take the same $\eta$ and $\zeta$ as in \eqref{eq0201}.	
Note that
$\int_{B_R(x_0)} \eta^4 \, dx$ is comparable to the volume of $B_R$,
$$
\int_{B_R(x_0)} \eta^4 \, dx = N_1 R^d
$$
for some constant $N_1$ independent of $R$.
As a test function, multiply both sides by $(u-\bar{u}(t)) \eta^4(x) \zeta^2(t)$,
where
$$
\bar{u}(t) = \left(\int_{B_R(x_0)}\eta^4\,dx\right)^{-1}\int_{B_{R}(x_0)} u \eta^4 \, dx
= \frac 1 {N_1} \dashint_{B_R(x_0)}u \eta^4 \, dx.
$$
Since
$$
\int_{B_R(x_0)} u_t (u-\bar u(t))\eta^4\zeta^2\,dx=\int_{B_R(x_0)} (u_t-\bar u(t)_t) (u-\bar u(t))\eta^4\zeta^2\,dx,
$$
we get
\begin{align*}
\frac 1 2&\int_{B_R(x_0)} (u-\bar u(t))^2 \eta^4(x) \zeta^2(t) \, dx
+ \int_{t_0-R^2}^{t} \int_{B_R(x_0)} A_{ij}(s,x,u) D_j u D_i((u-\bar u(t)) \eta^4 \zeta^2)\\
&=  \int_{t_0-R^2}^{t} \int_{B_R(x_0)} b(s,x,u,\nabla u)(u-\bar u(t)) \eta^4 \zeta^2
- a_i(s,x,u) D_i((u-\bar u(t)) \eta^4 \zeta^2)\\
&\,\,\,+  \int_{t_0-R^2}^{t} \int_{B_R(x_0)} (u-\bar u(t))^2 \eta^4 \zeta \zeta_t.
\end{align*}
From this,
\begin{align*}
&\frac 1 4\sup_{t \in (t_0-R^2,t_0)}\int_{B_R(x_0)} (u-\bar u(t))^2(t,x) \eta^4(x) \zeta^2(t) \, dx
+ \frac \mu 2 \int_{Q_R(X_0)} |\nabla u|^2 \eta^4 \zeta^2\\
&\le
4\int_{Q_R(X_0)} |A_{ij}| |\nabla u| |u-\bar u(t)| \eta^3 |\nabla \eta| \zeta^2
+ \int_{Q_R(X_0)} |a_i| |\nabla u| \eta^4 \zeta^2\\
&\,\,\,+ 4\int_{Q_R(X_0)} |a_i| |u-\bar u(t)| |\nabla \eta| \eta^3 \zeta^2
+ \int_{Q_R(X_0)} |b| |u-\bar u(t)| \eta^4 \zeta^2
+ \int_{Q_R(X_0)} (u-\bar u(t))^2 \eta^4 |\zeta_t| \zeta\\
&:= J_1 + J_2 + J_3 + J_4 + J_5.
\end{align*}
Again we estimate each term by using Young's inequality.

\noindent
Estimate of $J_1$:
$$
J_1 \le \frac \mu {16} \int_{Q_R(X_0)} |\nabla u|^2 \eta^4 \zeta^2
+ N \int_{Q_R(X_0)} |u-\bar u(t)|^2 |\nabla \eta|^2 \eta^2 \zeta^2.
$$

\noindent
Estimate of $J_2$:
\begin{align*}
J_2 &\le N\int_{Q_R(X_0)} |u|^{\gamma/2} |\nabla u| \eta^4 \zeta^2
+ N\int_{Q_R(X_0)} |f| |\nabla u| \eta^4 \zeta^2\\
&\le \frac \mu {16} \int_{Q_R(X_0)} |\nabla u|^2 \eta^4 \zeta^2
+ N \int_{Q_R(X_0)} |u|^{\gamma} \eta^4 \zeta^2
+ N \int_{Q_R(X_0)} |f|^2 \eta^4 \zeta^2.
\end{align*}

\noindent
Estimate of $J_3$:
\begin{align*}
J_3 &\le N\int_{Q_R(X_0)} |u|^{\gamma/2} |u-\bar u(t)||\nabla \eta| \eta^3 \zeta^2
+ N\int_{Q_R(X_0)} |f| |u-\bar u(t)| |\nabla \eta| \eta^3 \zeta^2\\
&\le N\int_{Q_R(X_0)} |u|^{\gamma} \eta^4 \zeta^2
+ N\int_{Q_R(X_0)} |u-\bar u(t)|^2 |\nabla \eta|^2 \eta^2 \zeta^2
+ N\int_{Q_R(X_0)} |f|^2 \eta^4 \zeta^2.
\end{align*}

\noindent
Estimate of $J_4$:
\begin{align*}
J_4 &\le N\int_{Q_R(X_0)} |\nabla u|^{2(1-1/\gamma)} |u-\bar u(t)| \eta^4 \zeta^2
+ N\int_{Q_R(X_0)} |u|^{\gamma-1} |u-\bar u(t)|\\
&\quad+ N\int_{Q_R(X_0)} |g| |u-\bar u(t)|\\
&\le \frac \mu {16}\int_{Q_R(X_0)} |\nabla u|^2 \eta^4 \zeta^2
+ N\int_{Q_R(X_0)} (|u|^\gamma+|u-\bar u(t)|^{\gamma})
+ N\int_{Q_R(X_0)} |g|^{\frac \gamma {\gamma-1}}.
\end{align*}
Hence
\begin{align*}
&\sup_{t \in (t_0-R^2,t_0)}\int_{B_R(x_0)} (u-\bar u(t))^2 \eta^4 \zeta^2 \, dx
+ \mu \int_{Q_R(X_0)} |\nabla u|^2 \eta^4 \zeta^2\\
&\,\,\le NR^{-2} \int_{Q_R(X_0)} |u-\bar u(t)|^2\eta^2 \zeta
+ N\int_{Q_R(X_0)} (|u|^{\gamma}+|u-\bar u(t)|^{\gamma})\\
&\quad+ N\int_{Q_R(X_0)}(|f|^2+|g|^{\frac \gamma {\gamma-1}})\\
&\,\,:= N_0(I_1 + I_2 + I_3).
\end{align*}
Note that
\begin{align*}
\int_{B_R(x_0)}& |u-\bar{u}(t)|^{2} \, dx
= \int_{B_R(x_0)} \left|u(t,x) - \frac 1 {N_1} \dashint_{B_R(x_0)} u(t,y) \eta^4(y) \, dy \right|^{2} \, dx\nonumber\\
&= \int_{B_R(x_0)} \left|\frac 1 {N_1} \dashint_{B_R(x_0)} u(t,x) \eta^4(y) \, dy - \frac 1 {N_1} \dashint_{B_R(x_0)} u(t,y) \eta^4(y) \, dy \right|^{2} \, dx\nonumber\\
&\le N_2 \int_{B_R(x_0)} \dashint_{B_R(x_0)} |u(t,x) - u(t,y)|^{2} \, dy \, dx.
\end{align*}
Therefore, the term $I_1$ can be estimated exactly as before by using the Sobolev-Poincar\'e inequality instead of the boundary Sobolev-Poincar\'e inequality:
$$
I_1
\le \frac \mu {16N_0} \sup_{t \in (t_0-R^2,t_0)} \int_{B_{R}(x_0)} |u-\bar u(t)|^2 \eta^4 \zeta^2 \, dx+N\left(R^{-q}\int_{Q_R(X_0)}|\nabla u|^{q}\right)^{\frac {2}{q}},
$$
where
$$
q=\frac{2(d+2)}{d+4}\in (1,2).
$$

The only difference is in the estimate of $I_2$, which we focus on below.
First, we note that, by the triangle inequality,
$$
I_2\le N\int_{Q_R(X_0)}|u-\bar u(t)|^\gamma\eta^4\zeta^2+N\int_{Q_R(X_0)}|\bar u(t)|^\gamma\eta^4\zeta^2:=I_{21}+I_{22}.
$$
We estimate $I_{21}$ in the same way as the term $I_2$ in the boundary case:
\begin{align*}
I_{21}&\le N\sup_{t \in (t_0-R^2,t_0)} \left(\int_{B_{R}(x_0)} |u-\bar u(t)|^2 \, dx\right)^{\frac{2}{d}}
\int_{Q_R(X_0)} |\nabla u|^2 \, dx \, dt\\
&\le N\sup_{t \in (t_0-R^2,t_0)} \left(\int_{B_{R}(x_0)} |u|^2 \, dx\right)^{\frac{2}{d}}
\int_{Q_R(X_0)} |\nabla u|^2 \, dx \, dt.
\end{align*}
For $I_{22}$, by the triangle inequality we have
$$
I_{22}\le N\int_{Q_R(X_0)}|\bar u(t)-c|^\gamma
+NR^{d+2}c^\gamma:=I_{221}+I_{222},
$$
where
$$
c=\dashint_{\!\!t_0-R^2}^{\,\,\,t_0} \bar u(t)\,dt.
$$
The estimate of $I_{222}$ is simple: since $\gamma q/2>1$, by H\"older's inequality,
$$
I_{222}\le NR^{d+2}\left(\dashint_{Q_R(X_0)}|u|\right)^\gamma\le NR^{d+2}\left(\dashint_{Q_R(X_0)}|u|^{\gamma q/2}\right)^{2/q}.
$$
To estimate $I_{221}$, we use Poincar\'e's inequality in $t$ to get
$$
I_{221}\le NR^{d+2} \left(\int_{t_0-R^2}^{t_0}|\bar u_t|\,dt\right)^\gamma
\le NR^{d+2} \left(\int_{t_0-R^2}^{t_0} R^{-d} \left|\int_{B_R(x_0)} u_t(t,x)\eta^4\,dx\right|\,dt\right)^\gamma.
$$
It follows from the equation that
$$
I_{221}\le NR^{d+2} J^\gamma,
$$
where
$$
J=\int_{t_0-R^2}^{t_0}R^{-d}\left|\int_{B_R(x_0)} (D_i( A_{ij}(t,x,u) D_j u + a_i(t,x,u)) + b(t,x,u,\nabla u))\eta^4\,dx\right|.
$$
Integrating by parts gives
\begin{align*}
J&\le
NR^{2}\dashint_{Q_R(X_0)} |A_{ij}| |\nabla u|  \eta^3 |\nabla \eta|
+ NR^{2}\dashint_{Q_R(X_0)} |a_i| |\nabla \eta| \eta^3
+ NR^{2}\dashint_{Q_R(X_0)} |b|\eta^4\\
&\le NR \dashint_{Q_R(X_0)}(|\nabla u|+|u|^{\gamma/2}+|f|)+NR^{2}\dashint_{Q_R(X_0)}(|\nabla u|^{2(1-1/\gamma)} + |u|^{\gamma-1} + |g|)\\
&:=N(J_6+J_7+J_8+J_9+J_{10}+J_{11}).
\end{align*}
Now it remains to use H\"older's inequality on each term as follows:
\begin{align*}
J_6&=R \dashint_{Q_R(X_0)}|\nabla u|\le R \left(\dashint_{Q_R(X_0)}|\nabla u|^q\right)^{\frac {2}{\gamma q}}
\left(\dashint_{Q_R(X_0)}|\nabla u|^{2}\right)^{\frac 1 2-\frac 1 {\gamma}}\\
&=\left(\dashint_{Q_R(X_0)}|\nabla u|^q\right)^{\frac {2}{\gamma q}}
\left(\int_{Q_R(X_0)}|\nabla u|^{2}\right)^{\frac 1 2-\frac 1 {\gamma}},
\end{align*}
\begin{align*}
J_7&=R\dashint_{Q_R(X_0)}|u|^{\gamma/2}\le R \left(\dashint_{Q_R(X_0)}|u|^{\gamma q/2}\right)^{\frac 2 {\gamma q}}\left(\dashint_{Q_R(X_0)}|u|^{\gamma}\right)^{\frac 1 2-\frac 1 \gamma}\\
&=\left(\dashint_{Q_R(X_0)}|u|^{\frac {\gamma q}2}\right)^{\frac 2 {\gamma q}}\left(\int_{Q_R(X_0)}|u|^{\gamma}\right)^{\frac 1 2-\frac 1 \gamma},
\end{align*}
$$
J_8=R\dashint_{Q_R(X_0)}|f|\le R\left(\dashint_{Q_R(X_0)}|f|^2\right)^{\frac 1 2}\le
\left(\dashint_{Q_R(X_0)}|f|^2\right)^{\frac 1 {\gamma}}
\left(\int_{Q_R(X_0)}|f|^2\right)^{\frac  1 2-\frac 1 {\gamma}},
$$
\begin{align*}
J_9&=R^{2}\dashint_{Q_R(X_0)}|\nabla u|^{2(1-1/\gamma)}
\le R^2\left(\dashint_{Q_R(X_0)}|\nabla u|^q\right)^{\frac 2 {\gamma q}}
\left(\dashint_{Q_R(X_0)}|\nabla u|^2\right)^{1 -\frac 2 \gamma}\\
&=\left(\dashint_{Q_R(X_0)}|\nabla u|^q\right)^{\frac 2 {\gamma q}}
\left(\int_{Q_R(X_0)}|\nabla u|^2\right)^{1 -\frac 2 \gamma},
\end{align*}
\begin{align*}
J_{10}&=R^{2}\dashint_{Q_R(X_0)}|u|^{\gamma-1}
\le R^2\left(\dashint_{Q_R(X_0)}|u|^{\frac {\gamma q}2}\right)^{\frac 2 {\gamma q}}
\left(\dashint_{Q_R(X_0)}|u|^{\gamma}\right)^{1-\frac 2 {\gamma}}\\
&=\left(\dashint_{Q_R(X_0)}|u|^{\frac {\gamma q} 2}\right)^{\frac 2 {\gamma q}}
\left(\int_{Q_R(X_0)}|u|^{\gamma}\right)^{1-\frac 2 {\gamma}},
\end{align*}
\begin{align*}
J_{11}&=R^2\dashint_{Q_R(X_0)}|g|\le R^2\left(\dashint_{Q_R(X_0)}|g|^{\frac {\gamma} {\gamma-1}}\right)^{\frac {\gamma-1} \gamma}\\
&\le
\left(\dashint_{Q_R(X_0)}|g|^{\frac {\gamma} {\gamma-1}}\right)^{\frac 1 {\gamma}}
\left(\int_{Q_R(X_0)}|g|^{\frac {\gamma} {\gamma-1}}\right)^{1-\frac 2  \gamma}.
\end{align*}
Note that by the assumptions on $u$, $f$ and $g$, and \eqref{eqn:2.2},
$$
\int_{Q_R(X_0)} |\nabla u|^2,\quad
\int_{Q_R(X_0)} |f|^2,
\quad
\int_{Q_R(X_0)} |g|^{\frac{\gamma}{\gamma-1}},\quad
\int_{Q_R(X_0)}|u|^\gamma
$$
are uniformly bounded.
Thus by combining the estimates above together, we get
$$
\mu \dashint_{Q_R(X_0)} |\nabla u|^2 \eta^4 \zeta^2
\le N\left(\dashint_{Q_R(X_0)}|\nabla u|^{q}+|u|^{\frac {\gamma q} 2}\right)^{\frac 2 q}+N\dashint_{Q_R(X_0)}(|f|^2+|F|^2)
$$
$$
+N\sup_{t \in (t_0-R^2,t_0)} \left(\int_{B_{R}(x_0)} |u|^2 \, dx\right)^{\frac{2}{d}} \dashint_{Q_R(X_0)} |\nabla u|^2,
$$
where
$$
F = 
|g|^{\frac12\frac{\gamma}{\gamma-1}}.
$$
and $N$ depends on $d, \mu, \mu_1,\mu_2$ as well as
$\|u\|_{V_2(\cU_T)}$, $\|f\|_{L_2(\cU_T)}$ and $\|g\|_{L_{\gamma/(\gamma-1)}(\cU_T)}$.
This together with the estimate of $I_2$ yields
the inequality in the proposition.
\end{proof}

\subsection{Proof of Theorem \ref{thm3}}

To prove Theorem \ref{thm3}, we use the boundary and interior estimate proved in the previous subsections as well as the following result, which is a version of Proposition 1.3 in \cite{GiaStr}; also see Chapter V in \cite{Gi83}.
\begin{proposition}
								\label{parareve}
Let $\Phi \ge 0$ in $Q = (0,T) \times \Omega$ and satisfies with some constant $r > 1$
$$
\dashint_{Q_R(X_0)} \Phi^r \le N_0 \left(\dashint_{Q_{8R}(X_0)} \Phi \right)^r + N_0 \dashint_{Q_{8R}(X_0)} \Psi^r + \theta \dashint_{Q_{8R}(X_0)} \Phi^r
$$
for every $X_0 \in Q$ and $Q_{8R}(X_0) \subset Q$,
where $\theta \in (0,1)$,
then $\Phi \in L_{\text{loc}}^p(Q)$ for $p \in [r, r+\kappa)$
and
$$
\left( \dashint_{Q_R(X_0)} \Phi^p \right)^{1/p}
\le N \left( \dashint_{Q_{8R}(X_0)} \Phi^r \right)^{1/r} + \left( \dashint_{Q_{8R}(X_0)} \Psi^{p} \right)^{1/p}
$$
for all $Q_{8R} (X_0)\subset Q$,
where $N$ and $\kappa$ depends only on $d$, $r$, $\theta$, and $N_0$.
\end{proposition}

\begin{proof}[Proof of Theorem \ref{thm3}]
We extend $u$ to $Q := (0,T) \times \bR^d$ so that $u(t,x)=0$ if $x \in \bR^d \setminus \Omega$.
It is easily seen that $u \in V_2(Q)$. Also $f$ and $g$ are extended in a similar way.
Let $R < R_0/4$.

Let $X_0 = (t_0,x_0)\in Q$ such that $Q_{4R}(X_0) \subset Q$.
Then we have the following three cases:
$B_{4R}(x_0) \subset \Omega$, $B_{4R}(x_0) \cap \partial \Omega \ne \emptyset$, or  $B_{4R}(x_0) \cap \Omega = \emptyset$.
In the first case, by Proposition \ref{prop01} we have
\begin{multline}
								\label{eq0215}
\dashint_{Q_{R/2}(X_0)} (|\nabla u|^2+|u|^\gamma)\\
\le N \left(\dashint_{Q_{4R}(X_0)}|\nabla u|^{q}+|u|^{\frac {\gamma q} 2}\right)^{2/q}
+N\dashint_{Q_{4R}(X_0)}(|f|^2+|F|^2)
\\
+N\sup_{t \in (t_0-R^2,t_0)} \left(\int_{B_{4R}(x_0)} |u|^2 \, dx\right)^{\frac{2}{d}} \left(\dashint_{Q_{4R}(X_0)} |\nabla u|^2\right),	
\end{multline}
where $F = |g|^{\frac12\frac{\gamma}{\gamma-1}}$ and $N = N(d,\mu,\mu_1,\mu_2,\|u\|_{V_2(\cU_T)},\|f\|_{L_2(\cU_T)},
\|g\|_{L_{\gamma/(\gamma-1)}(\cU_T)})$.
For the second case, take $y_0 \in \partial \Omega$
such that $|x_0 - y_0| = \text{dis}(x_0,\partial \Omega)$.
We see that
$$
B_{R/2}(x_0) \subset B_R(y_0) \subset B_{2R}(y_0) \subset B_{4R}(x_0).
$$
This combined with the inequality in Proposition \ref{prop02} gives \eqref{eq0215}
with $N = N(d,\mu,\mu_1,\mu_2,\beta)$. In the third case, \eqref{eq0215} holds trivially.

Now due to Corollary \ref{cor0201}, we have a sufficiently small $R_0'>0$, which depends on $u$, $\|f\|_{L_\sigma(\cU_T)}$ and $\|g\|_{L_\tau(\cU_T)}$, such that
for all $0 < R \le R_0'$,
$$
N\sup_{t \in (t_0-R^2,t_0)} \left(\int_{B_{4R}(x_0)} |u|^2 \, dx\right)^{\frac{2}{d}}
< 1.
$$
Then by applying Proposition \ref{parareve} with $r = 2/q>1$,
$$
\Phi := |\nabla u|^q + |u|^{\frac{\gamma q}2},
\quad
\Psi := |f|^q + |F|^q
$$
for some $p \in \left(2, \min\{\sigma, \frac{2(\gamma-1)}{\gamma}\tau\}\right]$ and all $Q_{4R}(Y_0) \subset Q$, we have
$$
\dashint_{Q_{R/2}(Y_0)} |\nabla u|^p + |u|^{\frac{\gamma p}2}
\le N \left(\dashint_{Q_{4R}(Y_0)} |\nabla u|^2 + |u|^{\gamma}\right)^{\frac p 2}
+ N \dashint_{Q_{4R}(Y_0)} |f|^p + |F|^p,
$$
where $p$ and $N$ depend only on $d$, $\mu$, $\mu_1$, $\mu_2$, $\beta$, $u$, $\|f\|_{L_\sigma(\cU_T)}$ and $\|g\|_{L_\tau(\cU_T)}$.
Covering $(\varepsilon,T)\times\Omega$ with appropriate cylinders $Q_{R/2}(Y_0)$ such that $Q_{4R}(Y_0) \subset Q$ gives the desired result.

If the initial condition is zero, we extend $u$ to be zero for $t<0$ so that the extended function $u$ satisfies \eqref{eq0218} on $(-1,T)\times\Omega$.
\end{proof}

\section{Linear estimates}
                                \label{sec7}
To proceed to the proofs of Theorems \ref{thm2} and \ref{thm4}, we need $L_p$-estimates for linear elliptic and parabolic equations.
In this section, we consider the following linear parabolic equation
\begin{equation}
                                \label{eq22.18}
\left\{
  \begin{aligned}
    v_t - D_i(a_{ij} D_j v) &= D_i h_i + h \quad \text{in} \,\, \cU_T, \\
    v&=0 \quad \text{on} \,\, \partial_p \cU_T,
  \end{aligned}
\right.
\end{equation}
and present some $L_p$-estimates necessary to the proofs of our regularity results.

We assume that the leading coefficients $a_{ij}$ are merely measurable in $t$
and have small mean oscillations with respect to $x \in \bR^d$. To describe this assumption, we set
$$
a_R^{\#} = \sup_{1 \le i,j \le d}
\sup_{\substack{(t_0,x_0) \in \bR^{d+1}\\r \le R}}
\dashint_{\!t_0-r^2}^{\,\,\,t_0}\dashint_{B_r(x_0)}\dashint_{B_r(x_0)} |a_{ij}(s,x) - a_{ij}(s,y)| \, dx \, dy \, ds.
$$

Assume that $|a_{ij}(t,x)| \le \mu^{-1}$ and $a_{ij}(t,x)\xi_i\xi_j \ge \mu |\xi|^2$
for all $\xi \in \bR^d$ and $(t,x) \in \bR^{d+1}$.
Also we assume
\begin{assumption}[$\rho_1$]
								\label{liSMO}
There is a constant $R_1 \in (0,1]$ such that $a_{R_1}^{\#} \le \rho_1$.
\end{assumption}

We recall the following result in \cite{DK09}.
\begin{proposition}
								\label{prop7.1}
Let $\Omega$ be the whole space $\bR^d$, a half space, or a bounded Lipschitz domain.
Let $\fp\in  (2,\infty)$, $p\in [\fp/(\fp-1), \fp]$, $h_i \in L_{p}(\cU_T)$, and $h \in L_p(\cU_T)$.
Then there exist positive $\beta=\beta(d,\fp,\mu)$ and $\rho_1=\rho_1(d,\fp,\mu)$
such that under Assumption \ref{assump2} ($\beta$) and Assumption \ref{liSMO} ($\rho_1$),
there is a unique $v \in \cH_p^1(\cU_T)$
satisfying \eqref{eq22.18} and
\begin{equation*}
\|v\|_{\cH_p^1(\cU_T)} \le N \|h_i\|_{L_p(\cU_T)}
+ N \|h\|_{L_p(\cU_T)},
\end{equation*}
where $N = N(d,\mu,\fp, R_1,T, |\Omega|)$.
\end{proposition}

The proposition above was proved in \cite{DK09} with $\beta$ and $\rho_1$ depending on $p$. An interpolation argument shows that they can be chosen sufficiently small in terms of $\fp$, instead of $p$. Indeed, if we have the $\cH_{\fp}^1$
solvability of \eqref{eq22.18} for some $a_{ij}$ and $\Omega$, by the duality, the $\cH^1_{\fp/(\fp-1)}$ solvability follows. Then we apply Marcinkiewicz's  theorem to get the $\cH^1_p$ solvability for any $p\in [\fp/(\fp-1),\fp]$.

By using Proposition \ref{prop7.1}, we derive
\begin{theorem}
								\label{liRes}
Let $\sigma, q \in (1,\infty)$, $\fp \in (2,\infty)$, $h_i \in L_{\sigma}(\cU_T)$, and $h \in L_q(\cU_T)$.
Assume that $p:= \min \{ \sigma, q^* \}\in[\fp/(\fp-1), \fp]$.
Then there exist positive $\beta=\beta(d,\fp,\mu)$ and $\rho_1=\rho_1(d,\fp,\mu)$
such that under Assumption \ref{assump2} ($\beta$) and Assumption \ref{liSMO} ($\rho_1$),
there is a unique $v \in \cH_p^1(\cU_T)$
satisfying \eqref{eq22.18} and
\begin{equation}
                                    \label{eq22.20}
\|v\|_{\cH_p^1(\cU_T)} \le N \|h_i\|_{L_\sigma(\cU_T)}
+ N \|h\|_{L_q(\cU_T)},
\end{equation}
where $N = N(d,\mu,\sigma,\fp, R_1, T, |\Omega|)$.
\end{theorem}
\begin{proof}
We first prove the existence. Since the equation is linear, by Proposition \ref{prop7.1} we may assume $h_i\equiv 0,i=1,2,...,d$ and thus $p=q^*$. Also by using a partition of the unity, it suffices to consider the cases $\Omega=\bR^d$ and $\Omega=\bR^d_+$.

In the case that $\Omega=\bR^d$, let $w$ be the unique $W^{1,2}_q(\cU_T)$ solution to
$$
\left\{
  \begin{aligned}
    w_t-\Delta w&=h \quad \text{in } \cU_T, \\
    w&=0 \quad \text{on } \partial_p \cU_T,
  \end{aligned}
\right.
$$
By the parabolic Sobolev imbedding theorem, we know that $w\in \cH^1_{q^*}(\cU_T)$ and
\begin{equation}
                                \label{eq22.10}
\|Dw\|_{L_{q^*}(\cU_T)}\le N\|h\|_{L_q(\cU_T)}.
\end{equation}
Now by Proposition \ref{prop7.1} there is a unique solution $\hat w\in \cH^1_{q^*}(\cU_T)$ to
$$
\left\{
  \begin{aligned}
    \hat w_t-D_i(a_{ij}D_j \hat w)=D_i((a_{ij}-\delta_{ij})D_j w) \quad &\text{in } \cU_T, \\
    \hat w=0 \quad &\text{on } \partial_p \cU_T.
  \end{aligned}
\right.
$$
Moreover,
\begin{equation*}
\|D\hat w\|_{L_{q*}(\cU_T)}\le N\|Dw \|_{L_{q^*}(\cU_T)}\le N\|h\|_{L_q(\cU_T)}.
\end{equation*}
Clearly $v:=w+\hat w\in \cH^1_p(\cU_T)$ is a solution to \eqref{eq22.18} and satisfies \eqref{eq22.20}.

In the case that $\Omega=\bR^d_+$, let $w$ be the unique $W^{1,2}_q((0,T)\times\bR^d)$ solution to $w_t-\Delta w=\bar h$ in $(0,T)\times\bR^d$ with the zero initial condition, where $\bar h$ is the odd extension of $h$ with respect to $x_1$. Clearly $w=0$ on $\partial_p\cU_T$, and as before we know that $w\in \cH^1_{q^*}(\cU_T)$ and satisfies \eqref{eq22.10}. Now we argue as in the previous case and find the solution $v$ to the initial-boundary value problem.

Finally, the uniqueness follows from the uniqueness of $\cH^1_{\min(\sigma,q)}(\cU_T)$ solution stated in Proposition \ref{prop7.1}.
\end{proof}

\section{Proof of Theorem \ref{thm4}}
							\label{sec6}

Before the proof of Theorem \ref{thm4} we present the following two lemmas, which assert that solutions to \eqref{eq3.30} are globally bounded and have some H\"{o}lder regularity.

\begin{lemma}
								\label{lembd}
Under the same assumptions as in Theorem \ref{thm4}, we have								
$$
\|u\|_{L_{\infty}(\cU_T)} \le N(d, \mu, \mu_1,\mu_2, \sigma, \tau, \beta, u, \|f\|_{L_\sigma(\cU_T)}, \|g\|_{L_\tau(\cU_T)}, T, |\Omega|).
$$
\end{lemma}

\begin{proof}
We use Theorem V.2.1 in \cite{LSU}.
To apply this theorem we need to check that
$$
u \in L_{q_1}(\cU_T),
\quad  \text{for some}\,\,q_1 \in \left(2+\frac4d, 4+\frac8d\right).
$$
We also need to check that
\begin{align}
								\label{eq0217}
\left(A_{ij}(t,x,u)\xi_j + a_i(t,x,u)\right)\xi_i &\ge \frac{\mu}{2} |\xi|^2 - N(\mu,\mu_1) |u|^{\gamma} - |u|^2 \psi(t,x),\\	
								\label{eq0219}
- b(t,x,u,\xi)u &\le \frac{\mu}4|\xi|^2 + N(\mu,\mu_2) |u|^{\gamma} + |u|^2 \psi(t,x)	
\end{align}
for all $(t,x) \in \bR^{d+1}$, $|u| \ge 1$, and $\xi \in \bR^d$,
where
\begin{equation}
								\label{eq0222}
\psi\in L_{q_2}(\cU_T),
\quad
\text{for some}\,\,q_2 \in \left(\frac d 2 + 1, d+2\right).
\end{equation}

Observe that
\begin{align*}
a_i(t,x,u)\xi_i &\le \mu_1 |\xi| (|u|^{\gamma/2}+f)
\le \varepsilon \mu_1 |\xi|^2 + N(\varepsilon)\mu_1|u|^{\gamma} + N(\varepsilon)\mu_1|f|^2\\
&\le \varepsilon \mu_1 |\xi|^2 + N(\varepsilon)\mu_1|u|^{\gamma} + N(\varepsilon)\mu_1|u|^2|f|^2
\end{align*}
for $|u| \ge 1$.
By taking $\varepsilon = 1/(2\mu_1)$, we have
\begin{align*}
\left(A_{ij}(t,x,u)\xi_j + a_i(t,x,u)\right)\xi_i
&\ge \mu|\xi|^2 - \varepsilon \mu_1 |\xi|^2 - N(\varepsilon)\mu_1|u|^{\gamma} - N(\varepsilon)\mu_1|u|^2|f|^2\\
&= \frac{\mu}2 |\xi|^2 - N(\mu,\mu_1) |u|^\gamma - N(\mu,\mu_1) |u|^2 |f|^2.
\end{align*}
For the inequality \eqref{eq0219}, we have
\begin{multline*}
-b(t,x,u,\xi)u \le |b(t,x,u,\xi)||u|
\le \mu_2 \left( |u||\xi|^{2(1-1/\gamma)} + |u|^{\gamma} + |u| g \right)\\
\le \mu_2 \left( \varepsilon |\xi|^2 + N(\varepsilon)|u|^{\gamma} + |u|^2 g\right)
= \frac{\mu}4 |\xi|^2 + N(\mu, \mu_2)|u|^\gamma + N(\mu,\mu_2)|u|^2 g.
\end{multline*}
Let $p$ be the constant from Theorem \ref{thm3} and recall $\gamma = 2 + \frac4d$.
Now we take
$$
q_1 \in \left(2+\frac 4 d, 4+\frac8d\right)
\cap \left(\gamma, \frac{\gamma p}2\right],
\quad
q_2 \in \left(\frac d 2 + 1, d+2\right)\cap\left(\frac d 2 + 1, \frac{\sigma}2 \wedge \tau\right],
$$
$$
\psi = N(\mu,\mu_1)|f|^2 + N(\mu,\mu_2)g.
$$
Then we see that all the conditions in \eqref{eq0217} -- \eqref{eq0222} are satisfied.
In particular, we have $u \in L_{q_1}(\cU_T)$ by Theorem \ref{thm3}.
\end{proof}

Upon replying on the fact that the solution $u$ to the equation \eqref{eq3.30} is bounded,
we obtain the H\"{o}lder continuity of $u$ from Theorem V.1.1 in \cite{LSU}.

\begin{lemma}
								\label{lemHolder}
Under the same assumptions as in Theorem \ref{thm4},
we have
$$
|u|_{\nu/2,\nu; \cU_T}\le N,
$$
where $\nu\in (0,1)$ and $N>0$ depend on
$d$, $\mu$, $\mu_1$, $\mu_2$, $\sigma$, $\tau$, $\beta$, $u$, $\|f\|_{L_\sigma(\cU_T)}$, $\|g\|_{L_\tau(\cU_T)}$, $T$, and $|\Omega|$.
\end{lemma}

Let $\bar u (t,x)$ be an extension of $u(t,x)$ such that $\bar u (t,x) = 0$ if $x \in \bR^d \setminus\Omega$.
Then define
$$
a_{ij}(t,x) =
\left\{
\begin{array}{clc}
A_{ij}(t,x,\bar u(t,x)) &\quad \text{if}  \quad 0 \le t \le T,\\
\delta_{ij} &\quad \text{otherwise}.
\end{array}
\right.
$$
Also define
$$
h_i(t,x) := a_i(t,x,u(t,x)),
\quad
h(t,x) := b(t,x,u(t,x),\nabla u(t,x)).
$$
Then the equation \eqref{eq3.30} turns into
\begin{equation}							 \label{Linearized}
\left\{
  \begin{aligned}
    u_t - D_i(a_{ij} D_ju + h_i(t,x)) = h(t,x)
 	\quad &\text{in } \cU_T, \\
    u=0 \quad &\text{on } \partial_p \cU_T.
  \end{aligned}
\right.
\end{equation}
Note that
\begin{equation}
								\label{eq0212}
|h_i(t,x)| \le \mu_1( |u|^{\gamma/2} + f)
\le \mu_1( N + f) \in L_{\sigma}(\cU_T),
\end{equation}
where $N$ is from Lemma \ref{lembd},
and
\begin{equation}
								\label{eq0211}
|h(t,x)| \le \mu_2( |\nabla u|^{2(1-1/\gamma)}+|u|^{\gamma-1} + g).	
\end{equation}
The coefficients $a_{ij}$ in \eqref{Linearized} satisfy, for any $(t_0,x_0) \in \bR^{d+1}$,
\begin{align*}
&\dashint_{\!t_0-r^2}^{\,\,\,t_0}\dashint_{x,y\in B_r(x_0)}
| a_{ij}(s,x) - a_{ij}(s,y) | \, dx \, dy \, ds\\
&= \dashint_{\!t_0-r^2}^{\,\,\,t_0} 1_{0 \le s \le T} \dashint_{x,y \in B_r(x_0)}
| A_{ij}(s,x,\bar u(s,x)) - A_{ij}(s,y,\bar u(s,y)) | \, dx \, dy \, ds\\
&\le \dashint_{\!t_0-r^2}^{\,\,\,t_0} 1_{0 \le s \le T} \dashint_{x\in B_r(x_0)}
| A_{ij}(s,x,\bar u(s,x)) - A_{ij}(s,x,\bar u(t_0 \wedge T,x_0)) | \, dx \, ds\\
&\,+ \dashint_{\!t_0-r^2}^{\,\,\,t_0} 1_{0 \le s \le T} \dashint_{x,y \in B_r(x_0)}
| A_{ij}(s,x, \bar u(t_0 \wedge T,x_0)) - A_{ij}(s,y, \bar u(t_0 \wedge T,x_0)) | \, dx \, dy \, ds\\
&\,+ \dashint_{\!t_0-r^2}^{\,\,\,t_0} 1_{0 \le s \le T} \dashint_{y \in B_r(x_0)}
| A_{ij}(s,y, \bar u(t_0 \wedge T,x_0)) - A_{ij}(s,y, \bar u(s,y)) | \, dy \, ds\\
&\le 2\omega(N r^{\nu}) + \cA^{\#}_r.
\end{align*}
That is, by using the notation in Section \ref{sec7}, we have
$$
a_R^{\#} \le 2\omega(N R^{\nu})
+ \cA^{\#}_R.
$$
Then by Assumptions \ref{SMO} and \ref{Aconti} there exists $R_2\in (0,R_1]$ such that
\begin{equation}
								\label{eq0224}
a_{R_2}^{\#} \le 2\rho,
\end{equation}
where $R_2$ depends on
the function $\omega$, and the constants $N$ and $\nu$ in Lemma \ref{lemHolder}.

\begin{proof}[Proof of Theorem \ref{thm4}]
We set $\fp$ to be $\max\{\sigma,\tau^*\}$,
and fix
\begin{equation}
								\label{eq0223}
\beta = \beta(d, \fp, \mu),
\quad
\rho= \frac 1 2 \rho_1(d,\fp,\mu),								
\end{equation}
where $\beta(d,\fp,\mu)$ and $\rho_1(d,\fp,\mu)$ are those in Theorem \ref{liRes}.

By Theorem \ref{thm3} there exists $p_0 > 2$ such that
$u \in \cH_{p_0}^1(\cU_T)$.
If $p_0 \ge \min\{\sigma, \tau^*\}$, we immediately obtain \eqref{eq0216}.
Otherwise, we see that $u$ satisfies \eqref{Linearized}. By \eqref{eq0212} and \eqref{eq0211},
$h_i \in L_{\sigma}(\cU_T)$ and $h \in L_{q_1}(\cU_T)$, where
$$
q_1 = \min \left\{ \frac{\gamma}{2(\gamma-1)}p_0, \tau \right\}.
$$
Set $p_1 = \min \{ \sigma, q_1^* \}$.
Since $\left(\frac{\gamma}{2(\gamma-1)}p_0\right)^*$ can be taken arbitrarily large in the case that
$$
\frac{\gamma}{2(\gamma-1)}p_0 \ge d+2,
$$
we see that
$$
p_1 =
\left\{
\begin{aligned}
&\min\{\sigma, \tau^*\} \quad &\text{if}\quad\frac{\gamma}{2(\gamma-1)}p_0 \ge d+2,\\
&\min\{\sigma, \left(\frac{\gamma}{2(\gamma-1)}p_0\right)^*, \tau^*\}
\quad &\text{if}\quad\frac{\gamma}{2(\gamma-1)}p_0 < d+2.
\end{aligned}
\right.
$$
Note that $p_1 \in [\fp/(\fp-1),\fp]$.
Thus by Theorem \ref{liRes} along with \eqref{eq0224} and \eqref{eq0223} applied to \eqref{Linearized} we have $u \in \cH^1_{p_1}(\cU_T)$
and
\begin{equation*}
\|u\|_{\cH^1_{p_1}(\cU_T)}\le N(\|h_i\|_{L_\sigma(\cU_T)}+\|h\|_{L_{q_1}(\cU_T)}),	
\end{equation*}
where $N=N(d,\mu,\sigma, p_1, \fp, R_2, T, |\Omega|)$.
Bearing in mind the definitions of $h_i$ and $h$ as well as
using Theorem \ref{thm3},
we obtain
\eqref{eq0216}
unless
\begin{equation}
                                        \label{eq23.23}
\frac{\gamma}{2(\gamma-1)}p_0<d+2\quad
\text{and}
\quad
\left(\frac{\gamma}{2(\gamma-1)}p_0\right)^*< \min\{\sigma,\tau^*\}.
\end{equation}
In this case,
$$
p_1 = \left(\frac{\gamma}{2(\gamma-1)}p_0\right)^* = \frac{(d+2)p_0}{d+4-p_0}
>p_0,
$$
and by \eqref{eq0211} and the fact that $u \in \cH_{p_1}(\cU_T)$, we have
$$
h \in L_{q_2}(\cU_T),
\quad
q_2 = \min\left\{ \frac{\gamma}{2(\gamma-1)}p_1, \tau \right\}.
$$
Note that $q_2>q_1$. We define $p_2=\min\{\sigma,q_2^*\}>p_1$. Then \eqref{eq0216} is proved unless \eqref{eq23.23} holds with $p_1$ in place of $p_0$.
We repeat the above argument to obtain $p_3,p_4, \cdots$ with the recursion formula
$$
p_{k+1} = \frac{(d+2)p_k}{d+4-p_k},
\quad
k=0,1,2,\cdots.
$$
Since
$$
p_{k+1} - p_k \ge \frac{p_0(p_0-2)}{d+4},
$$
there has to be an integer $k_0$ such that $p = p_{k_0} = \min\{\sigma, \tau^*\}$.
Note that $p_k \le \fp$ for all $k=0,\cdots,k_0$.
This allows us to use Theorem \ref{liRes} in the above iteration process with the same $\beta$ and $\rho$ in \eqref{eq0223} for all $k=0,\cdots,k_0$.
\end{proof}
We remark that in order to run the bootstrap argument above it is crucial that the starting point $p_0$ is greater than $2$.

\section{Elliptic case}
							\label{elliptic-sec}
							
This section is devoted to the proofs of Theorems \ref{thm1} and \ref{thm2}. We use the same strategy as in the parabolic case. Since the argument is similar and in fact simpler, we only give an outline of it.
In this section, as in Theorem \ref{thm1} we set
$$
\gamma =
\left\{
\begin{aligned}
&\frac{2d}{d-2}, \quad d > 2,\\
&\text{any number bigger than $2$}, \quad d = 2.
\end{aligned}
\right.
$$

\subsection{Reverse H\"{o}lder's inequality: Boundary estimate}

The following is an elliptic version of Proposition \ref{prop02},
where the estimate concerns $u$ on a ball centered at a boundary point.

\begin{proposition}
								\label{prop2}
Let $u\in \mathring{W}^1_2(\Omega)$ be a weak solution to \eqref{eq3.24},
and $f\in L_2(\Omega)$ and $g\in L_{\frac{\gamma}{\gamma-1}}(\Omega)$.
Then, for any $x_0 \in \partial\Omega$ and $0 < R \le R_0$,
\begin{multline*}
\dashint_{\Omega_{R/2}(x_0)} (|\nabla u|^2+|u|^\gamma)
\le N \left( \dashint_{\Omega_R(x_0)} |\nabla u|^q \right)^{\frac{2}{q}}
+ N \dashint_{\Omega_R(x_0)} \left(|f|^2 + |F|^2 \right)\\
+ N R^{d\gamma\left(\frac{1}{\gamma}-\frac12\right) + \gamma}
\left(\int_{\Omega_R(x_0)} |\nabla u|^2\right)^{\frac{\gamma}{2}-1}
\left(\dashint_{\Omega_R(x_0)} |\nabla u|^2\right),
\end{multline*}
where $F = |g|^{\frac12\frac{\gamma}{\gamma-1}}$ and
$N = N(d,\mu,\mu_1,\mu_2,\beta)$.
\end{proposition}

\begin{proof}
Let $\eta_0\in C_0^\infty(B_1)$ be a function satisfying $0 \le \eta_0 \le 1$ and $\eta_0 \equiv 1$ on $B_{1/2}$. Let $R\in (0,1]$ and $x_0\in\partial\Omega$. Using a test function $u\eta^2$, where $\eta=\eta_0(R^{-1}(\cdot-x_0))$,
we have
$$	
\int_{\Omega_R(x_0)} A_{ij}(x,u) D_i(u \eta^2) D_j u
+ \int_{\Omega_R(x_0)} a_i(x,u) D_i(u \eta^2)
= \int_{\Omega_R(x_0)} b(x,u,\nabla) u \eta^2.
$$
That is,
$$
\int_{\Omega_R(x_0)} A_{ij}(x,u) \eta (D_i u) \eta (D_j u)
= -\int_{\Omega_R(x_0)} A_{ij}(x,u) 2 u \eta (D_i \eta) (D_j u)
$$
$$
-\int_{\Omega_R(x_0)} a_i(x,u) D_i(u \eta^2)
+ \int_{\Omega_R(x_0)} b(x,u,\nabla) u \eta^2
=: J_1 + J_2 + J_3.
$$
We estimate $J_1$, $J_2$ and $J_3$ by using Young's inequality.

\noindent
Estimate of $J_1$:
$$
J_1 \le \mu \int_{\Omega_R(x_0)} |\nabla u| |u| \eta |\nabla  \eta|
\le \frac \mu {16} \int_{\Omega_R(x_0)} \eta^2 |\nabla u|^2
+ N \int_{\Omega_R(x_0)} |u|^2 |\nabla \eta|^2.
$$

\noindent
Estimate of $J_2$:
\begin{align*}
J_2 &\le \mu_1 \int_{\Omega_R(x_0)} |u|^{\gamma/2} |\nabla u| \eta^2
+ \mu_1 \int_{\Omega_R(x_0)} |f| |\nabla u| \eta^2\\
&\quad +  \mu_1 \int_{\Omega_R(x_0)} |u|^{\gamma/2} |u| |\nabla \eta| \eta
+ \mu_1 \int_{\Omega_R(x_0)} |f| |u| |\nabla \eta| \eta\\
&\le \frac \mu {16} \int_{\Omega_R(x_0)}|\nabla u|^2  \eta^2 +
N \int_{\Omega_R(x_0)} |u|^{\gamma} \eta^2
+ N \int_{\Omega_R(x_0)} |f|^2 \eta^2
+ N \int_{\Omega_R(x_0)} |u|^2 |\nabla \eta|^2.
\end{align*}

\noindent
Estimate of $J_3$:
\begin{align*}
J_3 &\le \mu_2 \int_{\Omega_R(x_0)} |\nabla u|^{2(1-1/\gamma)} |u| \eta^2
+ \mu_2 \int_{\Omega_R(x_0)} |u|^{\gamma-1} |u| \eta^2
+ \mu_2 \int_{\Omega_R(x_0)}|g| |u| \eta^2\\
&\le \frac \mu {16} \int_{\Omega_R(x_0)} |\nabla u|^{2} \eta^2
+ \mu_2 \int_{\Omega_R(x_0)} |u|^{\gamma} \eta^2
+ \mu_2 \int_{\Omega_R(x_0)}|g|^{\frac \gamma {\gamma-1}} \eta^2.
\end{align*}
Thus,
\begin{align*}
\mu &\int_{\Omega_R(x_0)} |\nabla u|^2 \eta^2\\
&\le N \int_{\Omega_R(x_0)} |u|^2 |\nabla \eta|^2
+ N \int_{\Omega_R(x_0)} |u|^\gamma
+ N \int_{\Omega_R(x_0)} (|f|^2+|g|^{\frac \gamma {\gamma-1}})\\
&:= N(I_1 + I_2 + I_3).
\end{align*}

Now we get estimates for $I_1$ and $I_2$ as follows.

\noindent
Estimate of $I_2$: By the Sobolev-Poincar\'e inequality,
\begin{equation}
                                    \label{eq13.33}
I_2 \le
NR^{d\gamma\left(\frac{1}{\gamma}-\frac12\right)+\gamma}
\left(\int_{\Omega_R(x_0)} |\nabla u|^2\right)^{\frac{\gamma}{2}},
\end{equation}
where $N=N(d,\beta)$.

\noindent
Estimate of $I_1$: Again by the Sobolev-Poincar\'e inequality,
$$
I_1 \le \frac{1}{R^2}\int_{\Omega_R(x_0)} |u|^2
\le
NR^d \left( \dashint_{\Omega_R(x_0)} |\nabla u|^q \right)^{\frac{2}{q}},
$$
where $q = \frac{2d}{d+2}$ and $N=N(d,\beta)$.

Therefore,
$$
\mu \int_{\Omega_R(x_0)} |\nabla u|^2 \eta^2
\le N R^d \left( \dashint_{\Omega_R(x_0)} |\nabla u|^q \right)^{\frac{2}{q}}
+ N \int_{\Omega_R(x_0)} \left(|f|^2 + |F|^2 \right)
$$
$$
\quad + N R^{d\gamma\left(\frac{1}{\gamma}-\frac12\right) + \gamma}
\left(\int_{\Omega_R(x_0)} |\nabla u|^2\right)^{\frac{\gamma}{2}-1}
\left(\int_{\Omega_R(x_0)} |\nabla u|^2\right),
$$
where
$$
F =
|g|^{\frac12\frac{\gamma}{\gamma-1}}.
$$
Finally, we obtain the desired inequality in the proposition
by adding \eqref{eq13.33} to the above inequality and diving all terms by $R^d$.
\end{proof}

\subsection{Proofs of Theorems \ref{thm1} and \ref{thm2}}
We go through the same arguments for the parabolic case shown in the previous sections.
Especially, for the proof of Theorem \ref{thm1} we use Proposition \ref{prop2} in this paper and Theorem V.2.2 in \cite{Gi83}, the latter is an elliptic version of Proposition \ref{prop01}. We leave the details to the interested reader.



\bibliographystyle{plain}

\def\cprime{$'$}\def\cprime{$'$} \def\cprime{$'$} \def\cprime{$'$}
  \def\cprime{$'$} \def\cprime{$'$}

\end{document}